\documentclass[12pt,letterpaper,reqno]{amsart}

\newtheorem{con}{\bf Conjecture}

\newtheorem{lem}{\bf Lemma}

\newtheorem{thm}{\bf Theorem}
\newtheorem{corr}{\bf Corollary}

\newcommand{\p}{\mathbb{P}}

\usepackage{times} 
\usepackage[T1]{fontenc} 
\usepackage{mathrsfs} 
\usepackage[dvips]{graphicx}
\usepackage[pdftex]{hyperref}
\usepackage{hyperref}
\usepackage{latexsym,color}
\usepackage{amsmath,amsfonts,amsthm,amssymb,amscd} 
\input amssym.def 
\input amssym.tex

\begin{document}

\title[On a property of random-oriented percolation in a quadrant] 
{On a property of random-oriented percolation in a quadrant}

\author{Dmitry Zhelezov}
\thanks{Department of Mathematical Sciences, 
Chalmers University Of Technology and University of Gothenburg} 
\address{Department of Mathematical Sciences, 
Chalmers University Of Technology and University of Gothenburg,
41296 Gothenburg, Sweden} \email{zhelezov@chalmers.se}


\date{\today}

\begin{abstract}
	Grimmett's random-orientation percolation is formulated as follows. The square lattice is used to generate an oriented graph such that each edge is oriented rightwards (resp. upwards) with probability $p$ and  leftwards (resp. downwards) otherwise. We consider a variation of Grimmett's model proposed by Hegarty, in which edges are oriented away from the origin with probability $p$, and towards it with probability $1-p$, which implies rotational instead of translational symmetry. We show that both models could be considered as special cases of random-oriented percolation in the NE-quadrant, 	provided that the critical value for the latter is $\frac{1}{2}$. As a corollary, we unconditionally obtain a non-trivial lower bound for	the critical value of Hegarty's random-orientation model. The second part of the paper is devoted to higher dimensions and we show that the	Grimmett model percolates in any slab of height at least $3$ in $\mathbb{Z}^3$.   
\end{abstract}

\maketitle

\section{Introduction}
Random-oriented percolation was first introduced by G. Grimmett \cite{Gr1} and is defined as follows. Consider the square lattice $\mathbb{Z}^2$ and let each vertical edge be directed upwards with probability $p \in [0, 1]$ and downwards otherwise. Analogously, each horizontal edge is directed rightwards with probability $p$ and leftwards otherwise. Let $\theta_G(p)$ be the probability that there is a directed path from the origin to infinity. By coupling with the classical bond percolation, it is not hard to show that $\theta_G(\frac{1}{2}) = 0$ using, for example, Lemma~2.1 in \cite{Lin1}. There is also the obvious symmetry 
$\theta_G(p) = \theta_G(1-p)$, so it is natural to ask if $\theta_G(p) > 0$ for $p \neq 1/2$. This conjecture was raised by Grimmett \cite{Gr1} for the first time. The most significant advance so far was also made by Grimmett in \cite{Gr2} where he showed that percolation does occur if one adds a positive density of randomly directed arcs, so that the total probability of a directed arc being present is greater than one. Also, W. Xianyuan \cite{Xia1} proved the uniqueness of the infinite cluster in the supercritical phase. He also conjectured that $\theta_G(p)$ is strictly monotone on $[1/2, 1]$, an obviously much stronger version of Grimmett's conjecture. Both conjectures seem to be far from resolution. 

In this note we consider a slightly different model, proposed by P. Hegarty on MathOverflow \cite{Heg1}, which we will refer as the $H$-model hereafter. It is defined as follows: for each edge $e$ of the integer lattice $\mathbb{Z}^2$ assign a direction away from the origin with probability $p$, and towards the origin otherwise. We say that a directed edge from $x$ to $y$ is oriented {\em inwards} if $\|x\| > \|y\|$, {\em outwards} otherwise, with the usual Euclidean norm. We denote by $\theta_H(p)$ the corresponding probability that there exists an infinite directed path from the origin. Note that this model coincides with the one proposed by Grimmett if we consider percolation only in the North-East quadrant. So, we will use $\theta_{NE}(p)$ for the percolation probability of the latter without abuse of notation. We conjecture that the NE quadrant is big enough for 
random-oriented percolation to occur.  
\begin{con} \label{con:NE}
 For random-oriented percolation in the North-East quadrant, $\theta_{NE}(p) > 0$ for all $p > \frac{1}{2}$.
\end{con}
It might be possible to prove a weaker result that $\theta_G(p) > 0$ implies $\theta_{NE}(p) > 0$, but we were unable to do so. The analogous result is well known for ordinary bond percolation in two dimensions and may be proven, for example, using RSW theory.  

It is not hard to show, using standard circuit-counting arguments, that $\theta_H(p) = 0$ for $p < \frac{1}{\mu^2}$, where $\mu$ is the connective constant of the square lattice. From the other side, $\theta_H(p) > 0$ for $p > \vec{p}_c$ due to coupling with oriented percolation with the critical probability $\vec{p}_c$. It is proved that $\vec{p}_c < 0.6735$, \cite{BBS}, and believed that $\vec{p}_c \approx 0.6447$. The main question that arises is whether we observe critical phenomena for the $H$-model and if so, what is the critical probability? Unfortunately, such a property appears to be very hard to establish.   
\begin{con} \label{con:monotonicity}
 The probability function $\theta_H(p)$ is strictly monotone in $[0, 1]$.
\end{con}

In this note we prove 
\begin{thm} \label{prop:main}
 Suppose $\theta_{NE}(1-p) > 0$. Then $\theta_H(p) = 0$. 
\end{thm} 

 Together with Conjecture~\ref{con:NE}, this would imply that for the $H$-model the critical probability 
 does exist and it is equal to $\frac{1}{2}$. Also, we get the following result unconditionally. 
 
\begin{corr} 
 $\theta_H(p) = 0$ for $0 < p < 1 - \vec{p}_c$. 
\end{corr} 
 Inserting the upper bound for oriented percolation, we get that $\theta_H(p)=0$ for $p < 0.3265$, which is considerably better 
 than $\frac{1}{\mu^2} \approx 0.15$. It is worth noting that the crucial property of the $H$-model is its' 90-degree rotational symmetry which is absent in the Grimmett model. 
 
 At the end of the paper we consider the Grimmett model in higher dimensions and prove that it is always supercritical, even if confined to a thin slab.
\begin{thm} \label{prop:dGrimmett}
	The $3$-dimensional Grimmett model confined to the slab $\mathbb{Z}^2 \times \{-1, 0, 1\}$ percolates for any $p \in [0, 1]$.
\end{thm} 

  The proof of Theorem \ref{prop:dGrimmett} exploits criticality of the two dimensional Grimmett model that had already been shown in \cite{Gr2}. Though the result supports Grimmett's  original conjecture, it seems that the crucial phenomena occur in the case of random-oriented percolation confined to a quadrant, which probably exhibits a phase transition in all dimensions. At least we can say that for any fixed $d \geq 2$ the $H$-model as well as the NE-quadrant model do not percolate for sufficiently small $p > 0$ due to the standard path-counting argument, but of course they do percolate for $p > \vec{p}_c$.

\section{Percolation in the NE quadrant}
This section is devoted to the proof of Theorem~\ref{prop:main}.
The dual lattice is a copy of  $\mathbb{Z}^2$ translated by the vector $(1/2, 1/2)$, but orientation rules can be defined in two different ways: turning orientations in the original lattice clockwise or counterclockwise. We denote such dual lattices $\mathbb{Z}^{2u}_d$ and $\mathbb{Z}^{2d}_d$ and define them as follows. If the edge $e$ fails to have an orientation in direction $\alpha$, the dual edge $e_d$ has orientation $\alpha + \pi/2$ in $\mathbb{Z}^{2u}_d$ and $\alpha - \pi/2$ in $\mathbb{Z}^{2d}_d$. The corresponding dual NE quadrants we denote by $\mathbb{Q}^u_d$ and $\mathbb{Q}^d_d$.

As an example, to prevent percolation in the NE quadrant, there must be a directed path from $(x, -1/2)$ to $(-1/2, y)$ in $\mathbb{Q}^u_d$ and, equivalently, a directed path from $(-1/2, y)$ to $(x, -1/2)$ in $\mathbb{Q}^d_d$ for some $x, y > 0$. That partly explains the superscripts $u$ (up) and $d$ (down). 
Also, we denote by $\Lambda_n$ the $2n\times2n$ square box with the center at the origin and let $B^{+}_{m,n}(x)$ be the event that there exists a path from $(x, 1/2)$ to  $(1/2, y)$ in $\mathbb{Q}^{u}_d$ for some $y > 0$ which lies entirely within $\Lambda_m\setminus\Lambda_n$. For existence of a path which avoids $\Lambda_n$ we will write $B^{+}_{n}(x)$ and $B^{+}(x)$ for the unconstrained event. In other words, $B^{+}_{n}(x) = \cup_{m > n}B^{+}_{m,n}(x)$ and $B^{+}(x) = \cup_{n > 0}B^{+}_{n}(x)$. As usual, we will denote by $\partial \Lambda_n$ the vertex boundary of the box $\Lambda_n$, i.e.: the set of vertices that have neighbors both inside and outside $\Lambda_n$.

Hereafter we will assume that all paths are in $\mathbb{Z}^{2u}_d$ if it is not explicitly stated otherwise.

We start with a few auxiliary lemmas that explicitly exploit duality.
\begin{lem} \label{lem:nodual}
	If $\theta_{NE}(1-p) > 0$ then $\theta_{NE}(p) = 0$  
\end{lem}
\begin{proof}
	This lemma can be proven in exactly the same way Harris showed there is no bond percolation in the quadrant at $\frac{1}{2}$ in his seminal paper \cite{Harr}. The only observation we need is that $\mathbb{Z}^{2u}_d$ has percolation parameter $1-p$ if we fix its origin at some point $(x, -1/2)$.
Then, according to Lemma 5.2 of \cite{Harr}, since $\theta_{NE}(1-p) > 0$, with probability one there is an oriented path in $\mathbb{Q}^{u}_d$ from $(x, -1/2)$ to $(-1/2, y)$ for some $x, y > 0$ because any dual path started at the $x$-axis crosses the $y$-axis a.s. But this means that the NE-cluster in the original lattice is finite a.s.   
\end{proof}

\begin{lem} \label{lem:limit}
	Let $n > 0$ and $\theta = \theta_{NE}(1-p) > 0$. Recall that $B^{+}_n(x)$ denotes the event that there is a path in $\mathbb{Q}^{u}_d$ from $(x, 1/2)$ to $(1/2, y)$ outside the box $\Lambda_n$ for some $y > 0$. Then 
	$$
		\liminf_{x \to \infty} \p\{B^{+}_n(x)\} \geq \theta. 
	$$
\end{lem}
\begin{proof}
    For each dual configuration $\omega$, let $\omega'$ be the modification of $\omega$ such that all edges inside $\Lambda_n$ are directed outwards from the point $(1/2, 1/2)$. Let $N(\omega)$ be the number of points $(x, 1/2)$ such that $\omega \in B^{+}(x)$ but $\omega' \notin B^{+}(x)$. Finally, 	  let $A$ be the set of all configurations $\omega$, such that $N(\omega) = \infty$. We claim that $\p(A) = 0$. 
    
    Let us assume $\p(A) > 0$ for the sake of contradiction. For $\omega \in A$ and $x > 0$, conditions $\omega \in B^{+}(x)$ and $\omega' \notin B^{+}(x)$
imply existence of a path $\partial\Lambda_n \to (x-1/2, y)$ in the original NE-quadrant of $\mathbb{Z}^2$ for some $y > 0$. Indeed, since $\omega' \notin B^{+}(x)$ there must be a NE-path in the original lattice that blocks $(x, 1/2)$ from the $y$-axis in the dual (all other configurations would have probability zero). On the other hand, it can emanate only at the boundary of $\Lambda_n$, because otherwise outwards orientation in $\Lambda_n$ would have no effect on $B^{+}(x)$. But due to the fact that $N(\omega) = \infty$ we conclude that there must be arbitrarily long NE-paths in the original lattice, hence there exists an infinitely long one, implying $\theta_{NE}(p) \geq \p(A) > 0$ and contradicting Lemma~\ref{lem:nodual}.

    Now, defining $N_m(\omega)$ as above but counting only points $(x, 1/2)$ with $x > m$ we have 
    $$\lim_{m \to \infty} \p(\{\omega | N_m(\omega) > 0\}) = 0$$ 
    and thus 
    $$
    	\liminf_{x \to \infty} \p\{ B^{+}_n(x) \} = \liminf_{x \to \infty} \p\{ B^{+}(x) \} \geq \theta.
    $$ 
\end{proof}

\begin{corr} \label{corr:q}
	Consider the $H$-model with edge probability $p$. Suppose $\theta = \theta_{NE}(1-p) > 0$. Then, for any $n, d > 0$ there exist $0 < N < M_0 < M$ such that $M_0 - N > d$ and
	$$\p(B^{+}_{M, n}(x)) > \frac{\theta}{2}$$ 
for each $x \in [N, M_0]$.
\end{corr}
\begin{proof}
   According to Lemma~\ref{lem:limit} we can pick $N$ such that  $\p\{B^{+}(x) | \Lambda_n \, \mathrm{is\,\,blocked} \,\} > 2\theta/3$ whenever $x > N$. As $B^{+}_{n}(x) = \cup_{m > n}B^{+}_{m,n}(x)$ by definition, it remains to take $M_0 = N + d + 1$ and $M$ large enough to fulfill the desired conditions.  
\end{proof}

Now we iterate Corollary~\ref{corr:q} to extend the directed path in the following way. Consider the event $B^{O}(A)$ that there exists a directed path \begin{equation} \label{eq:path}
(A, 1/2) \to (1/2, B) \to (-C, -1/2) \to (1/2, -D) \to (E, 1/2)
\end{equation}

 for some $B, C, D, E > 0$, where each part if the path, apart from axis crossings, lies inside a single dual quadrant. See Figure~\ref{fig:circuits}.

\begin{lem} \label{lem:extend}
 For each $N > 0$ there exists $M > N$ such that 
	$$\p\{B^{O}(A) \,\,\, \mathrm{in} \,\,\, \Lambda_M \setminus \Lambda_N\} > \left(\frac{\theta(1-p)}{2}\right)^4$$ 
	for some $A \in [N, M]$.
\end{lem}
\begin{proof}
	We may choose $M_1 > M_0 > N_1 > N > 0$ such that $\p\{B^{+}_{M_1, N}(x) > \frac{\theta}{2}\}$ for all $x \in (N_1, M_0)$. Thanks to Corollary~\ref{corr:q}, then, we pick $M'_0, N_2$ (enlarging $M_0$ and, subsequently, $M_1$ if necessary) having
	$M_0 > M'_0 > N_2 > N_1 > 0$ such that $\p\{B^{+}_{M_0, N_1}(x) > \frac{\theta}{2}\}$ whenever $x \in (N_2, M'_0)$. This guarantees that the probability of a directed path  
	$(x, 1/2) \to (1/2, B) \to (-C, 1/2)$ is greater than $\theta^2(1-p)/4$. Indeed, consider three events for any $x \in (N_2, M'_0)$:
	\begin{enumerate}
	  \item There exists a directed path $(x, 1/2) \to (1/2, B)$ in  $\Lambda_{M_0} \setminus \Lambda_{N_1}$
	  \item The axis crossing edge has direction $(1/2, B) \to (-1/2, B)$
	  \item There exists a directed path $(-1/2, B) \to (-C, 1/2)$ in $\Lambda_{M_1} \setminus \Lambda_{N}$
	\end{enumerate}
	By construction, $M_0 > B > N_1$ and these three events are independent since they depend on disjoint edge sets. Thus, the probability that they occur simultaneously is greater than $\theta^2(1-p)/4$. Repeating this consideration twice, we get the claim of the corollary. Note, that here the rotational symmetry of the $H$-model comes into play: turning by $\frac{\pi}{2}$ each time we are able to construct the almost closed path with probability bounded away from zero.  
\end{proof}

The following lemma asserts that the probability of a closed dual path is also bounded away from zero and is crucial. For $A > 0$, denote by $B^{O}(A, A)$ the event that there exists a closed directed circuit in $\mathbb{Z}^{2u}_d$ of the form (\ref{eq:path}) starting and finishing at $(A, 1/2)$. We now claim
\begin{lem} \label{lm:main}
   For each $N > 0$ there exists $M > N$ such that 
   $$\p\{B^{O}(A, A) \,\,\, \mathrm{in} \,\,\, \Lambda_M \setminus \Lambda_N\} > \frac{1}{9}\left(\frac{\theta(1-p)}{2}\right)^8$$
   for some $A \in [N, M]$.
\end{lem}
\begin{proof}
 Let us pick $M, A > 0$ given by Lemma~\ref{lem:extend}, such that 
$$
\p\{ B^{O}(A)\,\,\, \mathrm{in} \,\,\, \Lambda_M \setminus \Lambda_N \} >  \left(\frac{\theta(1-p)}{2}\right)^4,
$$ 
and fix the point $A$. Note that if there is at least one $A \to B \to C \to D \to E$ path, the inner-most and outer-most paths are unique and well-defined. Among all such paths $(A, 1/2) \to (1/2, B) \to (-C, -1/2) \to (1/2, -D) \to (E, 1/2)$ we choose the inner-most path $P_{\mathrm{in}}$ and the outer-most path $P_{\mathrm{out}}$, denoting their endpoints by $E_{\mathrm{in}}$  and $E_{\mathrm{out}}$ respectively, see Figure~\ref{fig:circuits}.  Since the inner-most path always lies inside the outer-most one (though they may touch each other or even coincide), conditioning on the existence of least one $A \to B \to C \to D \to E$ path, at least one of three events must occur: $E_{\mathrm{in}} < A$, $E_{\mathrm{out}} > A$ or $B^{O}(A, A)$, whence 
$$
	\p\{E_{\mathrm{in}} < A\} + \p\{E_{\mathrm{out}} > A\} + \p(B^{O}(A, A)) \geq \left(\frac{\theta(1-p)}{2}\right)^4.
$$
If 
$$
\p(B^{O}(A, A)) \geq \frac{1}{3}\left(\frac{\theta(1-p)}{2}\right)^4,
$$
we are done, so we suppose 
$$
\p\{E_{\mathrm{in}} < A\} \geq \frac{1}{3}\left(\frac{\theta(1-p)}{2}\right)^4
$$
without loss of generality. Conditioning on the states of the enclosed edges (the dashed region on the left part of Figure~\ref{fig:circuits}), we ensure that the enclosing path is indeed inner-most. Now, recall the other dual lattice $\mathbb{Z}^{2d}_d$ directed opposite to $\mathbb{Z}^{2u}_d$. By symmetry, the probability for a clockwise oriented path $(A, -1/2) \to (-1/2, -B') \to (-C', -1/2) \to (-1/2, D') \to (E', 1/2)$ such that $E' < A$ is equal to $\p\{E_{\mathrm{in}} < A\}$ and any such path must cross $(A, 1/2) \to (1/2, B_{\mathrm{in}}) \to (-C_{\mathrm{in}}, -1/2) \to (1/2, -D_{\mathrm{in}}) \to (E_{\mathrm{in}}, 1/2)$ whenever $E_{\mathrm{in}} < A$. The states of the edges enclosed by the counter-clockwise path are independent of the existence of the clockwise path up to the first crossing point with the latter, and such a crossing produces a closed circuit in $\mathbb{Z}^{2u}_d$ (and, equivalently, in $\mathbb{Z}^{2d}_d$). Thus,
$$
  \p(B^{O}(A, A)) \geq \p\{\mathrm{two\,paths\,exist}, E_{\mathrm{in}} < A,  E'_{\mathrm{in}} < A\} \geq \frac{1}{9}\left(\frac{\theta(1-p)}{2}\right)^8.
$$       
The case 
$$
\p\{E_{\mathrm{out}} > A\} \geq \frac{1}{3}\left(\frac{\theta(1-p)}{2}\right)^4
$$
is completely analogous, but we condition on the states of the outer edges (see the right-hand side of Figure \ref{fig:circuits}).

\begin{figure*}[t] 
\centering
\includegraphics[width=1.0\textwidth]{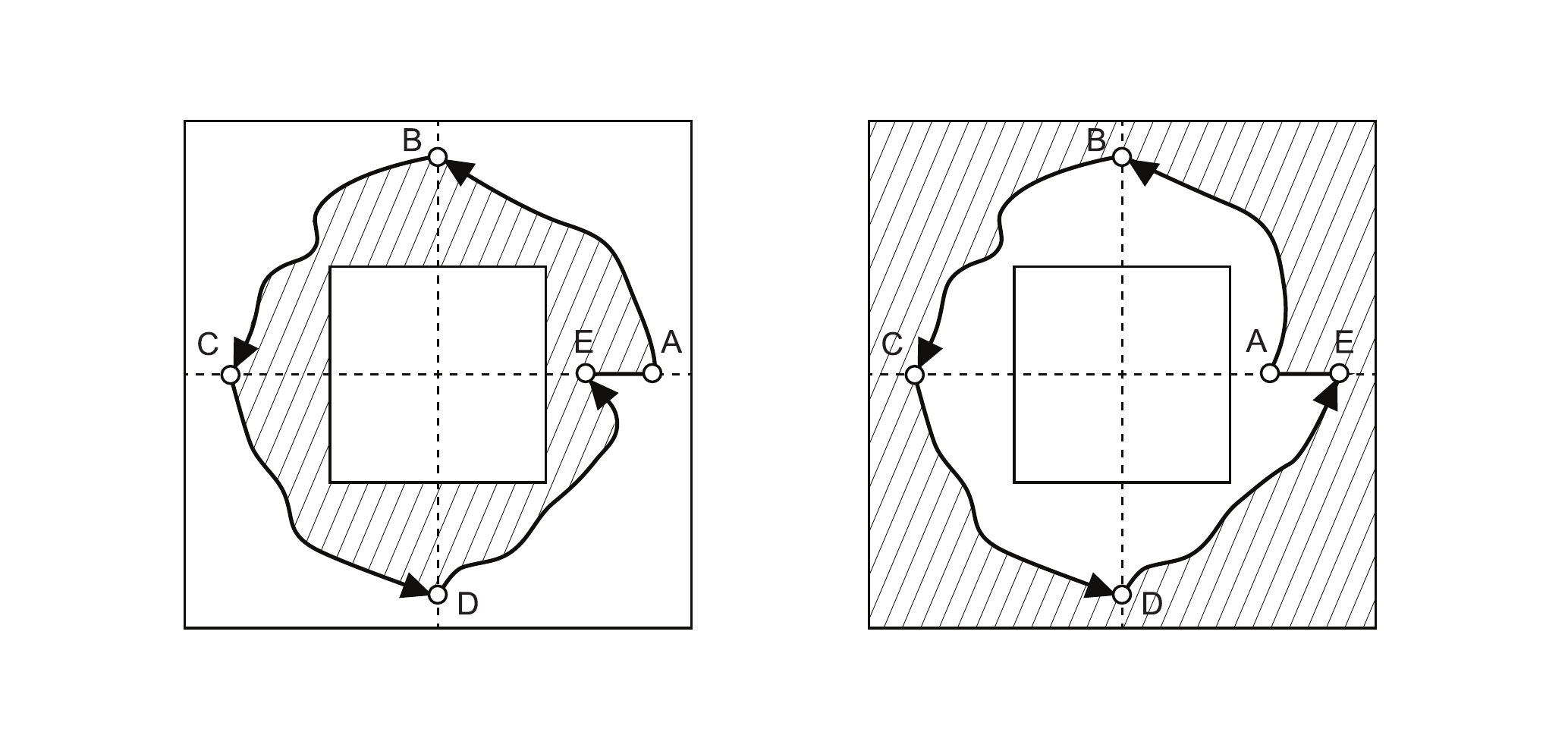}
\caption{At least one of three events must occur: $E_{\mathrm{in}} < A$ (on the left), $E_{\mathrm{out}} > A$ (on the right) 
or $B^{O}(A, A)$ (closed circuit). In the first and second cases we condition on the states of edges in the dashed region to apply a symmetry argument.}
\label{fig:circuits}
\end{figure*}
 
To finish the proof of Theorem~\ref{prop:main}, it simply remains to set up an infinite collection of frames provided by Lemma~\ref{lm:main} and apply the Borell-Cantelli lemma.

\end{proof}

\section{Higher dimensions}
In light of Theorem~\ref{prop:main} it is reasonable to guess that both the $H$- and Grimmett models are equivalent to percolation in the 
NE-quadrant. Indeed, suppose, say, $\theta_G(p) > 0$ and $p > 0$. According to Lemma~\ref{lem:nodual} either $\theta_{NE}(p) = 0$ or $\theta_{NE}(1-p) = 0$ (or maybe both). 
	Analogously, due to self-duality of the NW and the SE quadrants of the Grimmett model for all $p \geq 0$ it is easy to see that the percolation cluster restricted to any of these quadrants is almost surely finite. So, it is easy to believe that the infinite part of the percolation cluster (provided it is infinite) stays in the NE quadrant for $p > 1/2$, but we were unable to prove this. 

   The $d$-dimensional $H$- and Grimmett models differ significantly for $d \geq 3$: whereas the $H$-model probably remains equivalent to percolation in the quadrant, it is not difficult to show that $\theta^d_G(p) > 0$ for all $p$ when $d \geq 3$. In this section we will prove an even stronger result, namely that the Grimmett model percolates in any 3-dimensional slab of height at least three, as has already been announced in Theorem \ref{prop:dGrimmett}. On the other hand, the standard path counting argument implies that $\theta^d_H(p) = 0$ in all dimensions for sufficiently small $p > 0$ (which, of course, depends on $d$). 
 
 The idea of the proof of Theorem \ref{prop:dGrimmett} is that one can consider spatial orientated paths of the form 
	$$
	(x, y, 0) \to (x, y, 1) \to (x+1, y, 1) \to (x+1, y, 0) 
	$$
 as an additional arc $(x, y, 0) \to (x+1, y, 0)$ in the lattice $\mathbb{Z}^2 \times \{0\}$ and then apply the following theorem by Grimmett \cite{Gr2}. 
\begin{thm} {\rm (\textbf{Grimmett}, \cite{Gr2})} \label{thm:Gr2}
	Consider the following independent process on $\mathbb{Z}^2$ with parameters $a$ and $b$: rightward and leftward (respectively, upward and downward) arcs are placed independently between each pair of horizontal (respectively, vertical) neighbors. The probability of each upward or rightward arc being placed is $a$ and the probability of each downward or leftward arc being placed is $b$.
	 
	If $a + b > 1$ then the independent process with parameters $a, b$ contains an infinite oriented self-avoiding path from $0$ with strictly positive probability. 
\end{thm} 

 In the same paper \cite{Gr2} it is shown that the Grimmett model is equivalent to the independent process with parameters $p$ and $1-p$, which implies that if additional arcs introduced above would have been placed independently, the process is supercritical. Thus, the main technical difficulty to overcome is the dependence between such paths for neighboring vertices in $\mathbb{Z}^2 \times \{0\}$.

To do so, we will use Lemma 1.1 from \cite{LSS} to "bound" the dependent measure by a product measure from below.  
\begin{lem}  \label{lem:LSS}
	Suppose that $(X_s)_{s \in S}$ is a family of $\{0, 1\}$-valued random variables, indexed by a countable set $S$, with joint law $\nu$.
Suppose $S$ is totally ordered in such a way that, given any finite subset of $S$, $s_1 < s_2 < \ldots < s_j < s_{j+1}$,
and any choice of $\epsilon_1, \epsilon_2, \ldots, \epsilon_j \in \{0, 1\}$, then, whenever  
$\p(X_{s_1} = \epsilon_1, \ldots, X_{s_j} = \epsilon_j) > 0$,
\begin{equation} \label{ineq:marginal}
  \p(X_{s_{j+1}} = 1| X_{s_1} = \epsilon_1, \ldots, X_{s_j} = \epsilon_j) \geq \rho.
\end{equation}

Then $\nu$ stochastically dominates $\pi^{S}_{\rho}$, which is a product measure with parameter $\rho$. 
\end{lem}

\begin{proof} \textbf{(of Theorem \ref{prop:dGrimmett})}
	We may assume that the original Bernoulli percolation has parameter $p \in [1/2, \vec{p}_c]$ due to symmetry and obvious coupling with 2-dimensional directed percolation with the critical value $\vec{p}_c$. 
	Chess-color the lattice vertices of $\mathbb{Z}^2 \times \{0\}$ such that a vertex $(x, y, 0)$ is black if and only if $x + y$ is even. For a given black vertex $(x, y, 0)$ we consider oriented paths 
	\begin{equation} \label{eq:path1}
	(x, y, 0) \to (x, y, 1) \to (x \pm 1, y, 1) \to  (x \pm 1, y, 0) 
	\end{equation}
	and 
	\begin{equation} \label{eq:path2}
	(x, y, 0) \to (x, y, 1) \to (x, y  \pm 1 , 1) \to  (x, y  \pm 1, 0). 
	\end{equation}
	
	For white vertices the construction is completely similar but all oriented paths go through the plane $\mathbb{Z}^2 \times \{-1\}$. To make the distributions for white and black vertices identical and homogeneous we will consider a slightly different percolation model where
each edge that is not in $\mathbb{Z}^2 \times \{0\}$ has the same probability $1-p < p$ of being oriented in any direction. This can be achieved by coupling in the following way: sample a countable set of independent random variables uniformly distributed on $[0,1]$ for each edge in $\mathbb{Z}^2 \times \{-1, 0, 1\}$ except the plane $\mathbb{Z}^2 \times \{0\}$. To make orientations distributed according to the original law we assign rightwards (resp. upwards) orientation to the $i$th edge if $Y_i > 1-p$ and leftwards (resp. downwards) otherwise. If we assign rightwards (resp. upwards) only if $Y_i \in [2-2p, 1-p]$ we will end up with the desired model with all orientations having the same probability that is dominated by the original one. 
	From now on we can claim that each auxiliary path is present with probability $(1-p)^3$. 
	
	Let us fix the point $(x, y, 0)$ for a moment and write $A_i$ for the event that the $i$-th oriented path is present where the paths of the form (\ref{eq:path1}) and (\ref{eq:path2}) emanating from $(x, y, 0)$ have been ordered in some way. We observe that  
	\begin{equation}
	\p\left(\bigcap_{i \in I} A_i\right) \geq \prod_{i \in I} \p(A_i) \label{eq:positiveness}
	\end{equation} 
	for any (finite) index set $I$. On the other hand, each event $A_i$ may be seen as an additional arc in $\mathbb{Z}^2 \times \{0\}$ oriented outwards from $(x, y, 0)$ and we end up with the probability measure $\nu$ on $\mathbb{Z}^2 \times \{0\}$ that corresponds to the set of arcs on $\mathbb{Z}^2 \times \{0\}$ enriched in this way. 
	
	Let us order the set of all possible additional arcs in $S = \mathbb{Z}^2 \times \{0\}$ in some way, say, alphabetically. Note that, for each unoriented edge, there are two arcs directed in opposite ways and we consider them separately. Given the ordered countable set of arcs we
assign a random variable $X_{s_i}$ having $X_{s_i} = 1$ if the arc $s_i$ has been added during the enrichment process described above and zero otherwise. Obviously, the random variables $X_{s_i}$ are not independent, but we are almost in the setting of Lemma \ref{lem:LSS} and it remains to show that inequality (\ref{ineq:marginal}) holds for some $\rho$. 
  
  Let us fix some $X_{s_i}$. First, we note that $X_{s_i}$ is dependent only on arcs that either emanate from or end at the same vertex as $s_i$. Thus we have only six other variables $X_{s_1}, \ldots, X_{s_6}$ upon which it depends. Moreover, due to positive association (\ref{eq:positiveness}) we have  
$$
	\p(X_{s_i} = 1| X_{s_1} = \epsilon_1, \ldots, X_{s_j} = \epsilon_j) \geq \p(X_{s_i}=1 | X_{s_1} = 0, \ldots, X_{s_6} = 0)
$$    	    
and it remains to bound the last probability from below. Without loss of generality we may assume that $s_i$ is the arc $(x, y, 0) \to (x+1, y, 0)$. For both probabilities $p_1, p_2$ for arcs $(x, y, 0) \to (x, y, 1)$ and $(x+1, y, 1) \to (x+1, y, 0)$ being present, conditionally on $X_{s_k} = 0, k = 1,...,6$, we have a lower bound $(1-p)p^3$ because horizontal arcs lying on the planes $\mathbb{Z}^2 \times \{\pm1\}$ are independent. Hence we can take $\rho = (1-p)^3p^6$, which is bounded away from zero since $p \in [1/2, \vec{p}_c]$. Thus, Lemma \ref{lem:LSS} applies and $\nu$ dominates the product measure $\pi_\rho$.
	
	But the process with the product measure $\pi_\rho$ corresponds to adding a positive density of additional arcs to $\mathbb{Z}^2 \times \{0\}$ independently at random and we immediately arrive at the setting of Theorem \ref{thm:Gr2}. Thus, the original percolation process on $\mathbb{Z}^2 \times \{-1, 0, 1\}$ is supercritical.

\end{proof} 

\section{Conclusion}
  Despite seemingly simple formulations, both Grimmett's and the $H$-model appear to be difficult to analyze, mainly because of the absence of 
Harris-FKG-type correlation inequalities. The main difficulty to overcome is that any reasonable event, such as connectivity, defined in the random-orientation model is not increasing, in contrast with usual percolation models. Note, that Reimer's inequality still holds, but usually it is less fruitful and difficult to apply. It looks probable that further progress on the models in question requires substantially new ideas or at least considerable refinements of results in classical percolation.
  
  In the proof of Theorem \ref{prop:main} we have shown how purely geometrical considerations based on rotational symmetry together with self-duality can be used even without joining paths by any kind of correlation inequalities. However, the conjecture that, for example, percolation in Grimmett's model implies percolation in the NE quadrant contains the Harris theorem (that there is no bond percolation at $\frac{1}{2}$ in $\mathbb{Z}^2$), and thus it is probably very non-trivial to prove. 
  
  The crux of Theorem \ref{prop:dGrimmett} consists of a few observations. First, as shown by Grimmett, the random-orientation process is equivalent to the independent process in which oriented arcs are placed independently. On the other hand, the process such that leftwards (resp. downwards) and rightwards (resp. upwards) orientations are independent and present with probabilities $p$ and $q$ respectively, dominates the same process with parameters $p' \leq p$ and $q' \leq q$. Thus, it is monotone with respect to $p$ and $q$ and most of the classical results apply, for example, Menshikov's exponential decay theorem. 
  
  Second, in the same paper \cite{Gr2}, Grimmett shows how self-duality (again without any correlation inequalities) together with exponential decay implies criticality (or, maybe, supercriticality) of the random-orientation model. Finally, in the proof of Theorem \ref{prop:dGrimmett} we use the general domination result to show that additional paths introduced by two copies of $\mathbb{Z}^2$, namely $\mathbb{Z}^2 \times \{-1\}$ and $\mathbb{Z}^2 \times \{1\}$, have positive density and thus the resulting process is supercritical. Again, what we actually prove is that Grimmett's model in $\mathbb{Z}^2 \times \{-1, 0, 1\}$ dominates the independent process in $\mathbb{Z}^2$ with some parameters $p'$ and $q'$ such that $p' + q' > 1$.

\section{Acknowledgments}
The author is very grateful to his supervisor Professor Peter Hegarty for proposing the question, helpful discussions and for reviewing drafts over and over again. The author also deeply thanks Professor Jeff Steif for very careful proof-reading and pointing out a mistake in an earlier version.

\end{document}